\documentclass{amsart}

\usepackage{amsmath,amssymb,amsfonts,enumerate,amsthm,graphicx}

\usepackage{amscd}
\input xy
\xyoption{all}
\headsep=1cm \numberwithin{equation}{section}
 \newtheorem{thm}{Theorem}[section]
 \newtheorem{cor}[thm]{Corollary}
 
 \newtheorem{prop}[thm]{Proposition}
 \theoremstyle{definition}
 \newtheorem{defn}[thm]{Definition}
 \newtheorem{example}[thm]{Example}
 \theoremstyle{remark}

 \numberwithin{equation}{section}
 
 \DeclareMathOperator{\Deg}{deg}

\DeclareMathOperator{\Tor}{Tor}

\DeclareMathOperator{\Reg}{reg}

\DeclareMathOperator{\A}{\alpha}

\DeclareMathOperator{\pd}{proj\,\, dim}

\newcommand{\gens}{\mathrm{Gens}}

\begin{document}
\bibliographystyle{amsplain}

\title[A note on the  regularity of products]
{A note on the  regularity of products}

\author{Seyed Hamid Hassanzadeh}

\address{Seyed Hamid Hassanzadeh:
Faculty of Mathematical Sciences and Computer, Tarbiat Moallem
University, Tehran, Iran and School of Mathematics, Institute for
Research in Fundamental Sciences (IPM), P.O. Box 19395-5746, Tehran,
Iran.}

\email{hamid@dmat.ufpe.br}


\author{Siamak Yassemi}
\address{Siamak Yassemi: Department of Mathematics, University of Tehran, Tehran, Iran,
and School of Mathematics, Institute for research in fundamental
sciences (IPM), P. O. Box 19395-5746, Tehran, Iran.}

\email{yassemi@ipm.ir}

\thanks{The research of Hassanzadeh was in part
supported by grant No. 89130112 from IPM}
\thanks{The research of Yassemi was supported
in part by a grant No. 89130214 from IPM}

\date{\today}



\date{\today}

\begin{abstract}

Let $S={\Bbb K}[x_1,\dots,x_n]$ denote a polynomial ring over a field $\Bbb K$. Given a monomial ideal $I$ and a finitely generated multigraded $M$ over $S$, we follow Herzog's method to construct a multigraded free $S$-resolution of $M/IM$ by using multigraded $S$-free resolutions of $S/I$ and $M$. The complex constructed in this paper is used to prove the inequality $\Reg(IM)\leq \Reg(I)+\Reg(M)$ for a large class of ideals and modules. In the case where $M$ is an ideal, under one relative condition on the generators which specially does not involve the dimensions, the inequality $\Reg(IM)\leq \Reg(I)+\Reg(M)$ is proven.
\end{abstract}

\maketitle

\section*{Introduction}
Throughout this paper $S={\Bbb K}[x_1,\dots,x_n]$ is a polynomial ring over a field $\Bbb K$. The Castelnuovo-Mumford regularity, $\Reg(M)$, is one of the most important invariants
of a finitely generated graded module $M$ over a polynomial ring $S$. Despite in  general  the regularity of a module can be doubly exponential
in the degrees of the minimal generators and in the number of the variables, \cite{CS} and \cite{MM}, there are several descriptions of the regularity of sum, intersection and products of ideals in term of each factor. A look on the  enormous works in this topic, for example \cite{Chan},\cite{CMT}, \cite{S},\cite{T},\cite{EHU},\cite{CoHe} shows the importance of finding a neat formula for the regularity of a combination of two ideals.

 Let $I$ and $J$ be two monomial ideals of $S$ and let ${\Bbb F}$ and ${\Bbb G}$ be the multigraded free $S$-resolutions of $S/I$ and $S/J$. 
 In \cite{He} Herzog constructs a multigraded free $S$-resolution of $S/(I+J).$ This resolution generalizes the Taylor resolution \cite{Ta}. 
 The complex constructed in this way is used to generalize results on the Castelnuovo-Mumford regularity that were obtained for square-free monomial 
 ideals by G. Kalai and R. Meshulam \cite{KM}.
  More precisely, Herzog declares the expected formula for the sum and intersection of monomial ideals $I$, $J$ of the polynomial ring $S$,
$$\Reg(I+J)\leq \Reg(I)+\Reg(J)-1,$$
$$\Reg(I\cap J)\leq \Reg(I)+\Reg(J).$$

The problem on the regularity of products of homogeneous ideals, even monomial ideals, is more complicated. There are several counterexamples, \cite{St}, \cite{T}, \cite{CoHe}, which show that the inequality $\Reg(IJ)\leq \Reg(I)+\Reg(J)$  does not hold in general. The regularity of two ideals or an ideal and an $R$-module is related to the regularity of tensor product of two modules, the work started by Sidman \cite{S} and continued by Conca and Herzog \cite{CoHe} who showed that $\Reg(IM)\leq \Reg(I)+\Reg(M)$ for a finitely generated graded $R$-module $M$ and a homogeneous ideal $I$ in the case where $\dim(S/I)\leq 1$. In \cite{Ca} Caviglia showed that $\Reg(M\otimes N)\leq \Reg(M)+\Reg(N) $ whenever $\dim(\Tor_1^S(M,N))\leq 1$, the regularity of Tor modules was subsequently studied in detail by Eisenbud, Huneke, and Ulrich in\cite{EHU}.

The aim of this paper is to determine some cases in which the inequality $\Reg(IM)\leq \Reg(I)+\Reg(M)$ is valid.
By changing the point of view, instead of considering the codimension of the homogeneous ideal $I$ or $\dim(\Tor_1^S(S/I,M))$, a relation between the variables participate in the minimal generating set of $I$ and those correspond to the minimal generating set of $M$ is studied.

For a homogeneous ideal $I$ (resp. a finitely generated multigraded $S$-module $M$) we define $\gens(I)$ (resp. $\gens(M)$) to be the variables participate in the minimal generating set of $I$ (resp. in the degrees of the minimal generating set of $M$).
Using the techniques in \cite{He}, it is shown that in the case where $I$ is a monomial ideal and  $\gens(I) \cap \gens(M)= \emptyset$ the Herzog's complex (generalized Taylor complex) provides a free resolution for $M/IM$ in term of those of $S/I$ and $M$. This resolution in turn shows that the inequality $\Reg(IM)\leq \Reg(I)+\Reg(M)$  is valid in this case (see Theorem \ref{Tmain}). For two homogeneous ideals $I$ and $J$ of $S$, we show that the condition $\gens(I) \cap \gens(J)= \emptyset$ implies that $I\cap J=IJ$. This is the case where the inequality of the regularity was already known. Trying to extend the desired inequality for ideals, it is shown that if $I$ and $J$ are two homogeneous (not necessarily monomial) ideals  in which  $|\gens(I) \cap \gens(J)|\le 1$, then $\Reg(IJ)\leq \Reg(I)+\Reg(J)$ (see Theorem \ref{Tregproduc}). Finally, an already known example of Conca and Herzog \cite[2.1]{CoHe} shows that the inequality $\Reg(IJ)\leq \Reg(I)+\Reg(J)$ is no longer generally valid if $|\gens(I) \cap \gens(J)|\ge 2$ (see Example\ref{Ecounter}).


\section{Main results}
Throughout $\Bbbk$ is a field and $S=\Bbbk[x_1,\cdots,x_n]$ is a polynomial ring, $M$ is a finitely generated multigraded ($\mathbb N^n$-graded) $S$-module. In his technical paper Herzog \cite{He} defines a new product between free $S$-modules. For the sake of a ready to hand definition we restate the construction of this product.

For a homogeneous element $m \in M$ of degree $(a_1,\cdots, a_n) \in \mathbb N^n$  the unique monomial in $S$ which has the same degree as $m$ is denoted by $u_m$. We define the set of gens of $M$, $\gens(M)$, as the set of $x_i \in  \{x_1,\cdots,x_n\}$ such that $x_i$ divides some $u_m$ where $m$ is a member of a minimal generating set of $M$. In addition, $\gens(M)=\emptyset$, if $M$ is generated by elements of degree zero.

\begin{defn}Let $F$ and $G$ be free $S$-modules with homogeneous basis $B$ and $C$, respectively. The *-product of $F$ and $G$, $F*G$ is the multigraded free $S$-module with a basis given by the symbols $f*g$ where $f \in B$ and $g \in C$, the multidegree of $f*g$ is defined to be $[u_f, u_g]$, the least common multiple of $u_f$ and $u_g$.
\end{defn}
Comparing to the ordinary tensor product, $F\bigotimes G$ is a free $S$-module with the basis $f\otimes g$ where $f \in B$ and $g \in C$ and $\Deg(f\otimes g)= \Deg(u_fu_g)$. Hence, $F*G$ and $F\bigotimes G$ are free $S$-modules of a same rank. Keeping in mind that $S$ is a domain and the set $\{f*g: f\in B~~\text{and}~~g\in C\}$ is a basis for $F*G$, one can see that the homogeneous multigraded map $j: F\bigotimes G\to F*G$ given by $j(f\otimes g)=\text{gcd}(u_f,u_g)f*g$ is a monomorphism.

A homogeneous multigraded map $\varphi$ of free $S$-modules $F$ and  $G$ with bases $B$ and $C$ is generally defined by $\phi(f)=\sum\limits_{g\in C}a_{\textit{\tiny{fg}}}u_{\textit{\tiny{fg}}}g$ where $f\in B$, $a_{\textit{\tiny{fg}}}$ is a member of the field $\Bbbk$ and $u_{fg}= u_f/u_g$ provided $u_f/u_g\in S$, otherwise $u_{fg}=0$.

Let $H$ be another free $S$-module with homogeneous basis $D$, the map $\varphi * \text{Id}(H): F*H\to G*H$ is defined by $$\varphi * \text{Id}(H)(f*h)=\sum\limits_{g \in C}a_{\textit{\tiny{fg}}}u_{\textit{\tiny{hfg}}}g*h,$$
where $f \in B$ and $h \in D$ and for any monomials $x,y,z$, $u_{\textit{\tiny{xyz}}}=[u_x,u_y]/[u_x,u_z]$. As well, the map $ \text{Id}(H)*\varphi: H*F\to H*G$ is defined by
$$\text{Id}(H)*\varphi(h*f)=\sum\limits_{g \in C}a_{\textit{\tiny{fg}}}u_{\textit{\tiny{hfg}}}h*g.$$

Now, let $(F_{\bullet},\varphi)$ and $(G_{\bullet},\psi)$ be two multigraded complexes of free $S$-modules. We define the complex $(F_{\bullet}*G_{\bullet}, \A)$ as follows: $(F_{\bullet}*G_{\bullet})_i=\bigoplus_{j+k=i}F_j*G_k$ and the differential $\A_i:(F_{\bullet}*G_{\bullet})_i\to (F_{\bullet}*G_{\bullet})_{i-1}$ satisfies the equation $$\A_i=\sum_{j+k=i}\varphi_j*\text{Id}(G_k)+(-1)^j\text{Id}(F_j)*\psi_k.$$

We are now ready to state and prove our main theorem. This theorem provides a free resolution for a product of a monomial ideal and a multigraded module in term of their given  free resolutions. This theorem encompasses the previous known result on the resolution of the product of monomial ideals.


\begin{thm}\label{Tmain} Let $I$ be a monomial multigraded ideal of $S$ and $M$ be a finitely generated multigraded $S$-module such that $\gens(I)\cap\gens(M)=\emptyset$. Let $F_{\bullet}$ and $G_{\bullet}$ be the minimal multigraded free resolutions of $S/I$ and $M$, respectively. Then $F_{\bullet}*G_{\bullet}$ is a multigraded free resolution of $M/IM$.
\end{thm}
\begin{proof} The proof goes along the same lines as that of \cite[2.1]{He}, we just mention the slight modifications which have to be done. In the first step of the proof, we use polarization for the ideal $I$ and assume that $I$ is squarefree. As well by \cite[Theorem 2.1]{BH2}, the multigraded $S$-module $M$ can be lifted to a multigraded $T$-module $N$ where $T$ is a polynomial ring over $S$, such that all shifts in the multigraded free $T$-resolution of $N$ are squarefree. The shifts of this multigraded free $T$-resolution are of the expected form; so that after specialization, the multigraded  free $T$-resolution becomes the multigraded free $S$-resolution of $M$. Therefore we may assume that $I$ and $M$ have squarefree free resolution. We continue to the proof as in \cite{He}.

Let $S/I$ and $M$ admit minimal multigraded free resolutions $F_{\bullet}:0\to F_p\to\cdots\to F_1\to S \to 0$ and $G_{\bullet}:0\to G_q\to\cdots\to G_1\to G_0 \to 0$, respectively, where $F_i$ , resp. $G_i$, is a multigraded free $S$-module with basis $B_i$, resp. $C_i$,  for all $0 \leq i \leq p$, resp. $0 \leq i \leq q$. The complex $\widetilde{G}_{\bullet}$ arisen from the first spectral sequence of the double complex $F_{\bullet}*G_{\bullet}$ is of the form:
$$\widetilde{G}_{\bullet}:0 \to \bigoplus_{g\in C_q}(S/I_g)g\to \cdots \to \bigoplus_{g\in C_1}(S/I_g)g\to \bigoplus_{g\in C_0}(S/I_g)g \to 0$$
where $I_g$ is an ideal generated by the monomials $[u,u_g]/u_g$ in which $u$ is a member of the generating set of $I$. Here is the point that makes this theorem more general. The fact that $\widetilde{G}_{\bullet}$ is acyclic, \cite{He}, in conjunction with the fact that the second spectral sequence converges shows that to know what is resolved by $F_{\bullet}*G_{\bullet}$ we have to know what $H_0(\widetilde{G}_{\bullet})$ is. We consider  the most right terms of $F_{\bullet}*G_{\bullet}$, that is  $F_1*G_0\bigoplus S*G_0\xrightarrow{\psi}S*G_0\to 0.$ Since $S$ is generated by $1$, the map $j$ induces the isomorphisms $S*G_1\cong S\bigotimes G_1$ and $S*G_0\cong S\bigotimes G_0$. The assumption that $\gens(I)\cap\gens(M)= \emptyset$ implies that the homogenous homomorphism $j:F_1*G_0\to F_1\bigotimes G_0$ is an isomorphism, since $j(f_1*g_0)=\text{gcd}(u_{{f}_1}, u_{{g}_0})f_1\otimes g_0=f_1\otimes g_0$ for all $f_1\in B_1$ and $g_0\in C_0$. Hence we have the following commutative diagram, where $\varphi$ is the map at the beginning of the complex $F_{\bullet}\bigotimes G_{\bullet}$.
$$
 \xymatrix{
 F_1*G_0\bigoplus S*G_1\ar[r]^{\psi}\ar[d]^j_{\cong}&S*G_0\ar[d]^j_{\cong}\\
 F_1\bigotimes G_0\bigoplus S\bigotimes G_1\ar[r]^{\varphi}&S\bigotimes G_0\\
 }
 $$
To see that this diagram is commutative, we just need to verify the image of $f_1*g_0$ for $f_1 \in B_1$ and $g_0 \in C_0$.
$$\psi(f_1*g_0)=a_{{{f}_1}1}u_{{{g}_0}{{f}_1}1}1*g_0=
a_{{{f}_1}1}([u_{g_0}, u_{f_1}]/[u_{g_0},1]) 1*g_0=
a_{{{f}_1}1}u_{f_1}1*g_0,$$
 recall that gcd$(u_{f_1}, u_{g_0})=1$. We then have $j(\psi(f_1*g_0))=a_{{{f}_1}1}u_{f_1}1\otimes g_0=
\varphi(j(f_1*g_0))=\varphi(f_1\otimes g_0)$, which shows that the above diagram is commutative. Therefore, $H_0(F_{\bullet}*G_{\bullet})=H_0(F_{\bullet}\bigotimes G_{\bullet})= S/I\bigotimes _S M\cong M/IM$, as desired.
\end{proof}

Regarding the above theorem, the main theorem of \cite{He} deals with the case where $M=S/J$ and $J$ is a monomial ideal. In this case $\gens(M)=\emptyset$, hence the condition $\gens(M)\cap\gens(I)=\emptyset$ is automatically satisfied and the above argument for determining $H_0(\widetilde{G}_{\bullet})$ becomes vacuous.

With a free resolution of the product in hand, we are now able to give an upper bound for the Castelnuovo-Mumford regularity and the projective dimension of products. For a graded $S$-module $L$, we set $M_i(L)$ to  be the highest shifts appears in the graded minimal $S$-free resolution of $L$. The castelnuovo-Mumford regularity of $M$ is defined as $\Reg(L):=\max\{M_i(L)-i: i \geq 0\}$.

\begin{cor}\label{Creg} Let $I$ be a monomial ideal of $S$ and $M$ be a f.g. multigraded $S$-module such that $\gens(M)\cap\gens(I)=\emptyset$. Then
\begin{itemize}
\item[(a)]$\pd(M/IM)\leq \pd(M)+ \pd(I)+1$; and
\item[(b)]$\Reg(IM)\leq \Reg(I)+\Reg(M).$
\end{itemize}
\end{cor}
\begin{proof}Part (a) is due to the fact that $F_{\bullet}*G_{\bullet}$ is acyclic  and has length $\pd(M)+\pd(S/I)+1$.

For (b), since $F_{\bullet}*G_{\bullet}$ is not probably the minimal free resolution of $M/IM$ we have that $M_i(M/IM)\leq$ the highest shift in $(F_{\bullet}*G_{\bullet})_i$. The highest shift in   $(F_{\bullet}*G_{\bullet})_i$ is less than or equal to $\max\limits_{j+k=i}\{ M_j(M),M_k(S/I)\}$, and so
\begin{align*}
\Reg(M/IM)= & \max\{M_i(M/IM)-i:i \geq 0\} \\
      \leq & \max_{j+k=i}\{ M_j(M)-j,M_k(S/I)-k\}\\
      \leq & \Reg(M)+\Reg(S/I).
\end{align*}

Now, the exact sequence $0\to IM \to M \to M/IM\to 0$ implies that $$\Reg(IM)\leq \max \{\Reg(M),\Reg(M/IM)+1\}\leq \Reg(M)+\Reg(S/I)+1=\Reg(M)+\Reg(I).$$
\end{proof}

One may apply Corollary \ref{Creg} for the case where $M=J$ is a monomial ideal to obtain the formula $\Reg(IJ)\leq \Reg(I)+\Reg(J)$ provided that $\gens(J)\cap\gens(I)=\emptyset$. Although this is the desired formula for the regularity of product of ideals, it is shown in Corollary \ref{Cintersection} of the following general proposition that under the condition $\gens(J)\cap\gens(I)=\emptyset$ one has $IJ=I\cap J$. Hence to make the inequality $\Reg(IJ)\leq \Reg(I)+\Reg(J)$ valuable, we will later reduce the condition on Gens(c.f. Theorem \ref{Tregproduc}).

\begin{prop}\label{Pextended}
Consider the polynomial ring $S=\Bbbk[x_1,\cdots,x_n] $. Let $1 \leq k <n$ be an integer, $R=\Bbbk[x_1,\cdots,x_k]$ and $R'=\Bbbk[x_{k+1,\cdots,x_n}]$. Suppose that $M$ and $N$ are two extended modules, that is, there are graded (not necessarily multigraded) $R$-modules  $M_1$ and $R'$-module $N_1$ such that $M=M_1\otimes_RS$ and $N=N_1\otimes_{R'}S$. Then
\begin{itemize}
\item[(a)] $\Tor_i^S(M,N)=0~~\text{for all }i\geq 1$; and
\item[(b)]$\Reg(M\otimes_SN)\leq \Reg(M)+\Reg(N)$.
\end{itemize}
\end{prop}
\begin{proof}To prove (a), let $F_{\bullet}$ be a $R$-free resolution of $M_1$. $F_{\bullet}\otimes_RS $  provides a $S$-free resolution for $M_1\otimes_RS=M$. To compute $\Tor_i^S(M,N)$, we consider the homology of the complex $(F_{\bullet}\otimes_RS)\otimes_SN $. Considering the natural isomorphisms
$(F_{\bullet}\otimes_RS)\otimes_SN\cong
(F_{\bullet}\otimes_RS)\otimes_S(N_1\otimes_{R'}S)\cong
 F_{\bullet}\otimes_R(N_1\otimes_{R'}S)\cong
 F_{\bullet}\otimes_R(S\otimes_{R'}N_1)\cong
  F_{\bullet}\otimes_R(R\otimes_{\Bbbk}R')\otimes_{R'}N_1)\cong
  F_{\bullet}\otimes_{\Bbbk}N_1 $
  , we have $$\Tor_i^S(M,N)=H_i((F_{\bullet}\otimes_RS)\otimes_SN)=H_i(F_{\bullet}\otimes_{\Bbbk}N_1)=\Tor_i^{\Bbbk}(M_1,N_1)=0$$
  the last equality holds, since $\Bbbk $ is a field.

  For (b), it is enough to notice that if $F_{\bullet}$ and $G_{\bullet}$ are $R$-free resolution and $R'$-free resolution of $M_1$ and $N_1$, respectively, then by part (a)  $F_{\bullet}\otimes_{\Bbbk}G_{\bullet}$ is a $S$-free resolution for $M\otimes_SN$. Now a similar computation as in the proof of the Corollary \ref{Creg}  yields the assertion.
\end{proof}
\begin{cor}\label{Cintersection}Let $I$ and $J$ be two homogeneous ideals of $S$ with $\gens(J)\cap\gens(I)=\emptyset$, then $I\cap J=IJ$.
\end{cor}
\begin{proof}By the same token as Proposition \ref{Pextended}, suppose that $I=I_1S$ and $J=J_1S$ where $I_1$ and $J_1$ are graded ideals of $R$ and $R'$, respectively. Considering the natural maps $I_1\to R$, $J_1\to R'$, $R\otimes_RS\to S$ and $R'\otimes_{R'}S\to S$ and the fact that $S$ is flat over $R$ and $R'$, one sees that $S/I\cong R/I_1\otimes_RS$ and $S/J\cong R'/J_1\otimes_{R'}S$. Now the result follows from Proposition \ref{Pextended} in conjunction with the fact that $\Tor_1^S(S/I,S/J)=I\cap J/IJ$.
\end{proof}


The next theorem demonstrates the inequality $\Reg(IJ)\leq \Reg(I)+\Reg(J)$ for a large class of ideals $I$ and $J$, namely when $\gens(J)\cap\gens(I)$ consists of at most one element. As it is shown in a forehead example this class of ideals is the largest  class satisfies the inequality of regularity in the sense of the number of elements of the set of the intersections of Gens.

\begin{thm}\label{Tregproduc}Let $I$ and $J$ be two homogeneous ideals of $S$ such that $\gens(J)\cap\gens(I)$ consists of at most one element. Then $\Reg(IJ)\leq \Reg(I)+\Reg(J)$
\end{thm}
\begin{proof}Set $A:=\gens(J)\cap\gens(I)$. The case where $A=\emptyset$ is an immediate consequence of Proposition \ref{Pextended} or Corollary \ref{Cintersection}.

With no loss of generality, assume that $A=\{x_1\}$, that $I=I_1S$ where $I_1$ is an ideal of $R=\Bbbk[x_1,\cdots,x_k]$ and that $J=J_1S$ where $J_1$ is an ideal of $R'=\Bbbk[x_1,x_{k+1}\cdots,x_n]$. Let $F_{\bullet}$ be a $R$-free resolution of $I_1$ and $G_{\bullet}$ be a $R'$-free resolution of $R'/J_1$. Then  $F_{\bullet}\otimes_RS$ and $G_{\bullet}\otimes_{R'}S$ are $S$-free resolutions of $I$ and $S/J$, respectively. Hence for all integer $i$,
\begin{gather*}
\Tor_i^S(I,S/J)= H_i((F_{\bullet}\otimes_RS)\otimes_S(G_{\bullet}\otimes_{R'}S))\cong H_i((F_{\bullet}\otimes_RS)\otimes_{R'}G_{\bullet})\\
\cong H_i(F_{\bullet}\otimes_R(R\otimes_{\Bbbk[x_1]}R')\otimes_{R'}G_{\bullet})\cong H_i(F_{\bullet}\otimes_{\Bbbk[x_1]}G_{\bullet})=Tor_i^{\Bbbk[x_1]}(I_1,R'/J_1)
\end{gather*}
The fact that $\Bbbk[x_1]$ has global dimension 1 implies the vanishing of these $\Tor$ modules for all $i\geq 2$.
 To see the vanishing of the first $\Tor$ modules, notice that $R$ is a flat $\Bbbk[x_1]$ module, hence the exact sequence $0\to I_1\to R\to R/I_1\to 0$ yields $\Tor_1^{\Bbbk[x_1]}(I_1,R'/J_1)=\Tor_2^{\Bbbk[x_1]}(R/I_1,R'/J_1)=0$.

The vanishing of all $\Tor$ modules shows that $(F_{\bullet}\otimes_RS)\otimes_S(G_{\bullet}\otimes_{R'})S$ is a free resolution of $I\otimes_SS/J=I/IJ$. Now, a similar calculation as in the proof of Corollary \ref{Creg} shows the assertion.
\end{proof}

The next example of Conca and Herzog \cite{CoHe} shows that the inequality $\Reg(IJ)\leq \Reg(I)+\Reg(J)$ is
no longer true if $\gens(J)\cap\gens(I)$ consists of two elements.
\begin{example}\label{Ecounter} Let $R=\Bbbk[x_1,x_2,x_3,x_4]$, $I=(x_2,x_3)$ and $J=(x_1^2x_2,x_1x_2x_3,x_2x_3x_4,x_3x_4^2)$. The minimal free resolution of $I$, $J$ and $IJ$ are $0\to R(-2)\to R^2(-1)\to 0$, $0\to R^3(-4)\to R^4(-3)\to 0$ and $0\to R(-8) \to R^5(-6)\bigoplus R^2(-7)\to R^{10}(-5)\bigoplus R(-6)\to R^8(-4)\to 0$, respectively. Hence, we have $\Reg(I)=1$, $\Reg(J)=3$ and $\Reg(IJ)=5>\Reg(I)+\Reg(J)$.

 Notice that $\gens(I)=\{x_2,x_3\}$ and $\gens(J)=\{x_1,x_2,x_3,x_4\}$, thus $\gens(J)\cap\gens(I)=\{x_2,x_3\}$.
\end{example}



\begin{thebibliography}{alpha}

\bibitem{BH2} W. Bruns, J. Herzog, \textit{On multigraded resolutions},
Math.Proc. Camb. Phil. Soc. \textbf{25} (1995),245--257.

\bibitem{Ca} G. Caviglia,  \textit{Bounds on the Castelnuovo-Mumford regularity of tensor products},  Proc. Amer. Math. Soc.  \textbf{135}  (2007),  no. 7, 1949--1957

\bibitem{CS}G. Caviglia, E. Sbarra, \textit{ Characteristic-free bounds for the Castelnuovo–Mumford regularity}, Compositio Math. \textbf{141}(2005), 1365–1373.


\bibitem{Chan} K. A. Chandler, \textit{Regularity of fhe powers of an ideal} Comm.
Algebra \textbf{25} (1997), no. 12,3773-3776.

\bibitem{CMT}M. Chardin, N. C. Minh,  N. V. Trung,\textit{ On the regularity of products and
intersections of complete intersections}, Proc. Amer. Math. Soc. \textbf{135} (2007)1597–-1606.

\bibitem{CoHe}A. Conca,J. Herzog \textit{Castelnuovo--Mumford regularity of products
of ideals}. Collect. Math. \textbf{54} (2003), no. 2, 137-152.


\bibitem{EHU} D. Eisenbud, C. Huneke and B. Ulrich, \textit{The Regularity of Tor
and Graded Betti Numbers}, Amer. J. Math.\textbf{128}, (2006), no. 3, 573--605.


\bibitem{He} J. Herzog,  \textit{A Generalization of the Taylor Complex Construction}, Comm.
Algebra, \textbf{35}(2007), no.5, 1747-—1756.

\bibitem{KM} G. Kalai, R. Meshulam,  (2006). \textit{Unions and intersections of Leray complexes}, J. Combin. Theory Ser. A \textbf{113} (2006), 1586--1592.

\bibitem{MM} E.Mayer, A. Meyer,\textit{ The complexity of the word problem for commutative semigroups and polynomial ideals}, Adv. Math. \textbf{46}, (1982), 305--329.

\bibitem{S} J. Sidman, \textit{ On the Castelnuovo-Mumford regularity of products of
ideal sheaves}, Adv. Geom. \textbf{2} (2002), no. 3, 219-229.

\bibitem{St} B. Sturmfels, \textit{Four counterexamples in combinatorial algebraic geometry}, J. Algebra \textbf{230} (2000), 282–-294.

\bibitem{Ta} D. K. Taylor, \textit{Ideals generated by monomials in an $R$-sequence}, Ph.D. thesis, Univ. Chicago, Chicago, IL, 1966

\bibitem{T} N.Terai,  \textit{ Eisenbud-Goto inequality for Stanley-Reisner rings} In: Herzog, J.,
Restuccia, G, eds. Geometric and Combinatorial Aspects of Commutative Algebra.
Lecture Note in Pure and Applied Mathematics  \textbf{217}, Marcel Dekker, 379–-391, (2001).


\end{thebibliography}
\end{document}